\documentclass[preprint, 12pt, reqno, draft]{elsarticle}

\usepackage{amssymb}
\usepackage{mathrsfs}
\usepackage{amsfonts}
\usepackage{amsthm}
\usepackage{amsmath, stackrel}
\usepackage{xcolor}
\usepackage{subcaption}
\usepackage[labelformat=parens,labelsep=quad, skip=3pt]{caption}
\usepackage{amscd,amssymb,color,latexsym}
\usepackage{setspace}
\usepackage{graphics}
\usepackage{epstopdf}
\usepackage{graphicx}
\usepackage{subcaption}
\usepackage[left]{lineno}

\numberwithin{equation}{section}

\newtheorem{theorem}{Theorem}
\newtheorem{proposition}{Proposition}
\newtheorem{lemma}{Lemma}
\newtheorem{corollary}{Corollary}[proposition]
\newtheorem{remark}{Remark}

\begin{document}

\begin{frontmatter}

\title{CDF of non--central $\chi^2$ distribution revisited. Incomplete hypergeometric type 
functions approach}

\author[A2]{Dragana Jankov Ma\v sirevi\'c} 
\address[A2]{Department of Mathematics, University of Osijek, 31000 Osijek, Croatia}
\ead{djankov@mathos.hr}

\author[A3,B1]{Tibor K. Pog\'any \corref{cor1}}
\address[A3]{Faculty of Maritime Studies, University of Rijeka, 51000 Rijeka, Croatia}
\address[B1]{Institute of Applied Mathematics, \'Obuda University, 1034 Budapest, Hungary}
\ead{poganj@pfri.hr}
\cortext[cor1]{Corresponding author}

\begin{abstract}
The cumulative distribution function of the non--central chi-square distribution 
$\chi_\nu'^2(\lambda),\, \nu\in\mathbb{R}^+$ possesses an integral representation 
in terms of a generalized Marcum $Q$--function. Regarding some already known--results, here we derive 
a simpler form of the cumulative distribution function for $\nu = 2n \in\mathbb{N}$ degrees of 
freedom. Also, we express these representations in terms of an incomplete Fox--Wright function 
${}_p\Psi_q^{(\gamma)}$ and the generalized incomplete hypergeometric functions concerning the 
important special cases as ${}_1\Gamma_1,\, {}_2\Gamma_1$ and ${}_2\gamma_1$. New identities are 
established between ${}_1\Gamma_1$ and ${}_2\Gamma_1$ as well.
\end{abstract}

\end{frontmatter} 

\allowdisplaybreaks 

\section{Introduction  and motivation}

If $X_1,X_2,\dots X_\nu$ are independent homoscedastic normal $\mathscr N(\mu_j, \sigma^2)$,
$\mu_j\in \mathbb R,\,j=\overline{1, \nu}$, $\sigma>0$ random variables (rv) defined on a
standard probability space $(\Omega, \mathfrak F, \mathsf P)$, then the rv
$\xi = X_1^2+ \dots + X_\nu^2$ has non--central $\chi^2$--distribution with
$\nu \in \mathbb N$ degrees of freedom (the size number of the sum) and with the non--centrality
parameter $\lambda = \mu_1^2 + \cdots+\mu_\nu^2 \geq 0$. The distribution of the rv $\xi$
is usually denoted by  $\chi_\nu'^2(\lambda)$ \cite[p. 433]{Johnson} and the appropriate probability
density function (PDF) can be expressed as \cite{Fisher, Johnson}
   \begin{equation} \label{00}
	    f_{\nu,\lambda}(x) = \dfrac12\, {\rm e}^{-\frac{x+\lambda}2}
	                         \left(\dfrac{x}{\lambda}\right)^{\frac{\nu-2}4}
	                         \,I_{\frac\nu2-1}(\sqrt{\lambda x}), \qquad x>0,
   \end{equation}
where $I_\eta$ denotes the modified Bessel function of the first kind of the order $\eta$, which 
has the power series definition \cite[p. 249, Eq. {\bf 10.25.2}]{NIST}
   \begin{equation} \label{Star}
	    I_\eta(z) = \sum_{k \geq 0} \dfrac1{\Gamma(\eta+k+1)\, k!}\,
			           \Big(\dfrac{z}2\Big)^{2k+\eta}, \qquad \Re(\eta)>-1,\, z \in \mathbb C\,.
	 \end{equation}
The non--central $\chi^2$--distribution is, indeed, one of the most applied distributions, having
application in some statistical tests \cite{Kamel}, in finance, estimation and decision theory, time series
analysis \cite{Arpad0, Scharf}, in mathematical physics \cite[p. 435]{Johnson} and among others in
communication theory in which case the appropriate cumulative distribution  function (CDF) is given
by \cite[p. 66, Eq. (1.1)]{DJM_MJOM}
   \begin{equation}\label{distribution}
      F_{\nu,\lambda}(x) = 1-Q_{\frac\nu2}(\sqrt{\lambda},\sqrt{x}),\qquad x>0,
   \end{equation}
where $Q_\mu$ stands for the generalized Marcum $Q$--function \cite{Arpad1}
   \begin{equation}\label{M1}
      Q_\mu(a,b) = \dfrac1{a^{\mu-1}}\int_b^\infty t^\mu {\rm e}^{-\frac{t^2+a^2}{2}}\,
			             I_{\mu-1}(at)\,{\rm d}t, \qquad a,\mu>0,\,b \geq 0;
   \end{equation}
in this case the non-centrality parameter $\lambda$ is interpreted as a signal--to--noise
ratio \cite[p. 435]{Johnson}.

Although, in general, $\nu$ can be a nonnegative real number \cite{Robert}, considering, obviously, only
the PDF \eqref{00}, without bearing in mind the degrees of freedom ancestry and the structure of the rv
$\xi \sim \chi_\nu'^2(\lambda)$, most of the authors have dealt with PDF and CDF in the case when
$\nu = n \in \mathbb N$ (see e.g. \cite{Johnson, Patnaik, Pearson, Sankaran, Temme}). Quite recently,
Brychkov derive a closed--form expression for the generalized Marcum $Q$--function
\cite[p. 178, Eq. (7)]{Brychkov} in terms of the complementary error function
${\rm erfc}(z)$  \cite[p. 160, Eq. {\bf{7.2.2}}]{NIST} which immediately implies a new formula
for CDF \eqref{distribution} in the case when $n\in\mathbb{N}$ is odd; in the case of even number of
the degrees of freedom Jankov Ma\v sirevi\' c derived \cite{DJM_MJOM} a novel expression for the
appropriate CDF \eqref{distribution} in terms of the modified Bessel function of the second kind
$K_\nu$ and its incomplete variant $K_\nu(z,w)$, see \cite[p. 26, Eq. (1.30)]{Agrest}, where,
for $\Re(z)>0$, $K_\nu(z,w)\to K_\nu(z)$, as $w \to \infty$, in pointwise sense. It is also worth
to mention that Jankov Ma\v sirevi\'c established the computational efficiency (compare
\cite[Section 3]{DJM_MJOM}) of hers formulae versus the relations by Temme for even $n \in \mathbb N$,
which ones, rewritten in our setting read \cite[p. 58, Eq. (2.8)]{Temme}
   \begin{equation}\label{CDF_Temme}
      F_{n,\lambda}(x) = \left\{\begin{array}{ccl}
         1 - \dfrac12\left(\dfrac x\lambda\right)^{\frac{n}4} \left[T_{\frac{n}2-1}(\sqrt{\lambda x},\omega)
				   - \sqrt{\dfrac{\lambda}{x}}\,T_{\frac{n}2}(\sqrt{\lambda x},\omega)\right], &x >\lambda\\
             \dfrac12\left(\dfrac x\lambda\right)^{\frac{n}4}\left[\sqrt{\dfrac{\lambda}{x}}\,
						 T_{\frac{n}2}(\sqrt{\lambda x},\omega) - T_{\frac{n}2-1}(\sqrt{\lambda x},\omega)\right],
						& x< \lambda
        \end{array}\right.;
   \end{equation}
here $\omega = (x+\lambda)\,(2\sqrt{\lambda x}\,)^{-1} -1$, while
   \[T_\nu(\sqrt{\lambda x},\omega) = \int_{\sqrt{\lambda x}}^\infty
	                                    {\rm e}^{-(\omega+1)t}I_\nu(t)\,{\rm d}t\,. \]
Temme claimed that his formulae have certain computational advantages.\footnote{We
mention that the formulae \eqref{CDF_Temme} are also listed in the book Johnson {\it et al.} \cite[p. 441,
Eq. (29.20)]{Johnson}, but unfortunately in an erroneous form.}

Introducing the function
   \begin{equation} \label{functionS}
	    S_\nu(\sqrt{\lambda x},\omega) = \int_0^{\sqrt{\lambda x}} {\rm e}^{-(\omega+1)t}I_\nu(t)\,{\rm d}t,
	 \end{equation}
with the help of the Laplace transform of the modified Bessel function
\cite[p. 313, Eq. 2.15.3.1]{Prudnikov2} we conclude
   \[ S_\nu(\sqrt{\lambda x},\omega)+T_\nu(\sqrt{\lambda x},\omega) =
	          \dfrac{(x+\lambda-|x-\lambda|)^\nu}{\big(2 \sqrt{\lambda x}\,\big)^{\nu-1} |x-\lambda|},
						\qquad \nu>-1,\, \min\{x, \lambda\}>0, \]
it follows, by the corresponding considerations in \cite{DJM_MJOM}, that \eqref{CDF_Temme} can be
written in the modified symmetric form also for all $n \in \mathbb N$ and $\min\{x, \lambda\}>0$: 
  \begin{equation} \label{CDF_Temme2}
      F_{n,\lambda}(x) = \dfrac12\left(\dfrac x\lambda\right)^{\frac{n}4}
				     \Big\{ S_{\frac{n}2-1}(\sqrt{\lambda x},\omega) -
						 \sqrt{\dfrac{\lambda}{x}}\, S_{\frac{n}2}(\sqrt{\lambda x},\omega) \Big\},
	\end{equation}
which is a more convenient representation for numerical calculations. In turn, for
$x = \lambda$ this expression reduces to a difference of two $S_\nu(\lambda, 0)$ integrals 
which are generalized hypergeometric ${}_2F_2$ functions, {\it viz.}
   \begin{align} \label{CDF=}
	    F_{n,\lambda}(\lambda) &= \dfrac1{\Gamma(\frac{n}2+1)} \Big(\dfrac{\lambda}2\Big)^{\frac{n}2}\,
						 \Big\{ {}_2F_2 \Big[ \begin{array}{c} \frac{n-1}2, \frac{n}2\\ \frac{n}2+1, n-1 \end{array}
						 \Big| - 2 \lambda \Big] \notag \\
					&\qquad - \dfrac{\lambda}{n+2}\, {}_2F_2 \Big[ \begin{array}{c} \frac{n+1}2,
						 \frac{n}2+1\\ \frac{n}2+2, n+1 \end{array} \Big| - 2 \lambda \Big] \Big\}\,.
	 \end{align}
One of the main aims of this paper is to derive another elegant expression for $F_{2n,\lambda}$ related
to \eqref{CDF_Temme2}. That result is presented in the next section. In the Section 3 we show that this
new CDF formula can be explicitly expressed in terms of the incomplete Fox--Wright function
${}_p\Psi_q^{(\gamma)}$. The Section 4 consists from expressions for CDF which are established in
terms of the incomplete confluent ${}_1\Gamma_1$ and Gaussian hypergeometric function ${}_2\Gamma_1$.
Some new identities between ${}_1\Gamma_1, {}_2\Gamma_1$ and ${}_2\gamma_1$ end this part. 
The exposition closes the fifth section with a discussion and further related remarks.

\section{On CDF of $\chi_{2n}'^2(\lambda)$--distribution regarding Temme's result} \label{sec1}

In this section we will show that $F_{2n,\lambda}$ can be presented in simple form, containing only
modified Bessel functions $I_n$, $n\in\mathbb N_0$ and the function $S_\nu$ of the order $\nu=0$.
The main tool we refer to is a formula by Jankov Ma\v sirevi\' c which consists Theorem~\ref{pomoc}
exposed in a slightly different, but more condensed way then in \cite{DJM_MJOM}.

\begin{theorem} {\rm \cite[p. 4, Theorem 2.1]{DJM_MJOM}} \label{pomoc}
The {\rm CDF} of the non--central chi-square distribution with an even number of the degrees of freedom can
be represented in the form
   \begin{equation}\label{DJM1001}
      F_{2n,\lambda}(x) = {\rm e}^{-\frac{\lambda+x}{2}}\, \sum_{k \geq n}
			  \left(\dfrac x\lambda\right)^{\frac{k}2} \, I_k(\sqrt{\lambda x}),
   \end{equation}
where $n\in\mathbb{N}$ and $\min\{\lambda,x\}>0$.
\end{theorem}

In 1971 Agrest and Maksimov \cite[p. 139, Eq. (6.15)]{Agrest} concluded that
  \begin{equation}\label{AM}\int_0^z{\rm{e}}^{\mp \alpha t}I_0(t)\,{\rm d}t
	   = \dfrac1{\sqrt{\alpha^2-1}}\left\{1-{\rm{e}}^{\mp\alpha z}\left[I_0(z)
		 + 2 Y_2\left(\dfrac{z}{c},z\right) \pm 2 Y_1\left(\dfrac{z}{c},z\right)\right]\right\},
  \end{equation}
where the parameters $\alpha, c$ are related as $2 \alpha = c+c^{-1}$ and $Y_\nu(w,z)$ stands
for the Lommel function of two variables of the order $\nu$, defined by the Neumann series
\cite[p. 138, Eq. (6.5)]{Agrest}
   \[ Y_\nu(w,z)=\sum_{k\ge0}\left(\dfrac{w}{z}\right)^{\nu+2k}I_{\nu+2k}(z). \]

\begin{theorem}\label{theorem1}
For all $n\in\mathbb{N}$ and $\min\{\lambda,x\}>0$ there holds
   \begin{equation}\label{usporedba2}
      F_{2n,\lambda}(x)\!= \!\dfrac12 - {\rm e}^{-\frac{\lambda+x}{2}}
			                    \left(\frac12 I_0\left(\sqrt{\lambda x}\right)\! +\!
													\sum_{k=1}^{n-1}\left(\dfrac x\lambda\right)^{\frac{k}2}
                          I_k(\sqrt{\lambda x})\right)\!-\dfrac{|\lambda-x|}{4\sqrt{\lambda x}}
													S_0(\sqrt{\lambda x},\omega),
   \end{equation}
where $S_0$ is defined in \eqref{functionS} and $\omega=(x+\lambda)\,(2\sqrt{\lambda x}\,)^{-1} -1$.
\end{theorem}

\begin{proof}
Taking $z=\sqrt{\lambda x},\,c=\sqrt{\lambda/x}$ in \eqref{AM} we rewrite the expression in brackets into
   \begin{align}\label{help}
      I_0(\sqrt{\lambda x}) &+ 2Y_2\big(x,\sqrt{ \lambda x}\,\big) + 2Y_1\big(x,\sqrt{\lambda x}\,\big)\notag\\
			   &=  I_0(\sqrt{\lambda x}) + 2 \sum_{k \geq 0} \left(\dfrac x\lambda\right)^{k+1}
				     I_{2k+2}(\sqrt{\lambda x}) + 2 \sum_{k \geq 0}\left(\dfrac x\lambda\right)^{k+\frac12}
						 I_{2k+1}(\sqrt{\lambda x})\notag\\
         &= -I_0(\sqrt{\lambda x}) + 2 \sum_{k \geq 0}\left(\dfrac x\lambda\right)^k
				     I_{2k}(\sqrt{\lambda x}) + 2 \sum_{k \geq 0}\left(\dfrac x\lambda\right)^{k+\frac12}
						 I_{2k+1}(\sqrt{\lambda x})\notag\\
         &= -I_0(\sqrt{\lambda x}) + 2 \sum_{k \geq 0} \left(\dfrac x\lambda\right)^{\frac{k}2}
				     I_{k}(\sqrt{\lambda x}),
   \end{align}
where in the last equality we employed the elementary identity
   \[ \sum_{k \geq 0}a_k=\sum_{k\ge0}a_{2k}+\sum_{k\ge0}a_{2k+1}.\]
By simple considerations it follows from the relations \eqref{AM} and \eqref{help} that
   \[ \sum_{k \geq 0} \left(\dfrac x\lambda\right)^{\frac{k}2} I_{k}(\sqrt{\lambda x})
	      = \frac12 \Big({\rm e}^{\frac{\lambda+x}{2}} + I_0(\sqrt{\lambda x})\Big)
				- \dfrac{{\rm{e}}^{\frac{\lambda+x}{2}}|\lambda-x|}{4\sqrt{\lambda x}}
				  \int_0^{\sqrt{\lambda x}} {\rm{e}}^{-\frac{\lambda+x}{2\sqrt{\lambda x}}t} I_0(t)\,{\rm{d}}t.\]
We deduce the assertion of the theorem combining the above inferred formula with the identity
\eqref{DJM1001} given in Theorem~\ref{pomoc}.
\end{proof}

\section{CDF in terms of incomplete Fox--Wright function} \label{sec2}

In Theorem~\ref{theorem1} we derive a new representation for the CDF of rv $\xi \sim \chi_{2n}'^2(\lambda)$
in terms of the integral $S_0$. In this section, we derive new expression for $S_0$ which yields,
in combination with Theorem~\ref{theorem1}, a new form for the CDF.

The Fox--Wright generalized hypergeometric function with $p$ numerator parameters $a_1, \cdots, a_p$ and
$q$ denominator parameters $b_1,\cdots,b_q$ is defined by the series \cite[pp. 286--287]{Wright}
   \begin{equation}\label{FoxW1}
      {}_p\Psi_q \Big[\! \begin{array}{c} (a_1,A_1),\cdots, (a_p,A_p)  \\
                 (b_1,B_1),\cdots,(b_q,B_q) \end{array}\!\Big| z \Big]
		          = {}_p\Psi_q \Big[\! \begin{array}{c} ({\bf{a}}_p,{\bf{A}}_p)  \\
                 ({\bf{b}}_q,{\bf{B}}_q) \end{array}\!\Big| z \Big] 
							= \sum_{n \geq 0} \dfrac{ \prod\limits_{j=1}^p\Gamma(a_j+n A_j)}
							   {\prod\limits_{j=1}^q\Gamma(b_j+n B_j)} \dfrac{z^n}{n!},
   \end{equation}
where $A_i,\,B_j \ge 0$,  $i=1,\dots,p$, $j=1,\dots,q$. The defining series converges
in the whole complex $z$-plane when
   \[ \Delta := 1+ \sum_{j=1}^qB_j-\sum_{i=1}^pA_i>0;\]
when $\Delta = 0$ the series in \eqref{FoxW1} converges for $|z|<\nabla$, and $|z|=\nabla$ under
the condition $\Re(\mu)>1/2$ where
   \[ \nabla := \left( \prod\limits_{i=1}^pA_i^{-A_i}\right)\left(\prod\limits_{j=1}^qB_j^{B_j}\right),
	              \qquad \mu=\sum_{j=1}^q b_j-\sum_{i=1}^p a_i+\dfrac{p-q}{2}. \]
If in \eqref{FoxW1} we set $A_1 = \cdots = A_p = 1$ and $B_1= \cdots = B_q=1$ we get the generalized
hypergeometric function ${}_pF_q$, up to the multiplicative constant:
   \[ {}_p\Psi_q \Big[ \begin{array}{c} ({\bf{a}}_p,{\bf{1}})\\ ({\bf{b}}_q,{\bf{1}}) \end{array}
	               \Big|\, z \Big] = \dfrac{\Gamma(a_1)\cdots\Gamma(a_p)}{\Gamma(b_1)\cdots\Gamma(b_q)}\,
								 {}_pF_q \Big[ \begin{array}{c} {\bf{a}}_p\\ {\bf{b}}_q \end{array}\Big|\, z \Big]. \]
In what follows, the symbol ${}_p\Psi_q^{(\gamma)}$ stands for the incomplete Fox--Wright function
as a generalization of the complete Fox-Wright function ${}_p\Psi_q$, \cite{Srivastava}. The 
series definition reads \cite[p. 131, Eq. (6.1)]{Srivastava}
   \begin{equation}\label{IncFoxW1}
      {}_p\Psi_q^{(\gamma)} \Big[ \begin{array}{c} (\mu,M,x), ({\bf{a}}_{p-1},{\bf{A}}_{p-1})\\
                  ({\bf{b}}_q,{\bf{B}}_q) \end{array}\Big|\, z \Big]
								= \sum_{n \geq 0} \dfrac{\gamma(\mu+n M,x) \prod\limits_{j=1}^{p-1} \Gamma(a_j+n A_j)}
								  {\prod\limits_{j=1}^{q}\Gamma(b_j+n B_j)} \; \dfrac{z^n}{n!},
   \end{equation}
where $\gamma(a,x)$ denotes the lower incomplete gamma function
   \begin{equation}\label{LIG}
      \gamma(a,x)=\int_0^x{\rm{e}}^{-t}t^{a-1}\,{\rm{d}}t,\qquad \Re(a)>0.
	 \end{equation}
Parameters $M,A_j,B_j>0$ should satisfy the constraint
  \[\Delta^{(\gamma)} = 1+ \sum_{j=1}^qB_j-M-\sum_{j=1}^{p-1}A_j \geq 0,\]
while the convergence conditions coincide with the ones regarding the complete Fox--Wright \eqref{FoxW1}. 
The case $p = q = 1$ leads to the confluent incomplete Fox--Wright hypergeometric function.

\begin{lemma}
For all positive real numbers $\min\{\lambda,x\}>0$ there holds
   \[ S_\nu\Big(\sqrt{\lambda x},\frac{x+\lambda}{2\sqrt{\lambda x}}-1\Big) 
			          = \dfrac{2 (\lambda x)^{\frac{\nu+1}2}}{(x+\lambda)^{\nu+1}}\,
			            {}_1\Psi_1^{(\gamma)} \Big[ \begin{array}{c} (\nu+1,2,\tfrac{x+\lambda}2)  \\
                  (\nu+1,1) \end{array}\Big|\, \frac{\lambda x}{(x+\lambda)^2} \Big]\,.\]
\end{lemma}

\begin{proof}
By expanding the Bessel function in \eqref{functionS}, we obtain an infinite
series in terms of the lower incomplete gamma functions \eqref{LIG}:
   \begin{align*}
      S_\nu(\sqrt{\lambda x},\omega) &= \sum_{k \geq 0}\dfrac{2^{-(2k+\nu)}}{\Gamma(\nu+k+1)\,k!} 
			      \int_0^{\sqrt{\lambda x}} {\rm e}^{-(\omega+1)t}t^{2k+\nu}\,{\rm d}t\\
         &= \dfrac{1}{2^\nu(\omega+1)^{\nu+1}}\sum_{k\ge0}\dfrac{(2(\omega+1))^{-2k}}{\Gamma(\nu+k+1)\,k!} 
				    \int_0^{\frac{x+\lambda}{2}} {\rm e}^{-u}u^{2k+\nu}\,{\rm d}u\\
         &= \dfrac{1}{2^\nu(\omega+1)^{\nu+1}} \sum_{k \geq 0} \dfrac{\gamma(\nu+1+2k,(x+\lambda)/2)}
				    {\Gamma(\nu+k+1)\,k!\,(2(\omega+1))^{2k}}
   \end{align*}
which is equivalent to the statement, being $\omega = (x+\lambda)(2\sqrt{\lambda x})^{-1}-1$.
\end{proof}

The previous expression is very convenient for computing CDF $F_{n,\lambda}$ for not too large values 
of the variables. 

Concerning \eqref{CDF_Temme2} and Lemma 1 we deduce the following result.

\begin{theorem}
For all positive real numbers $\min\{\lambda,x\}>0$ there holds
   \begin{align*}
      F_{n,\lambda}(x) &= \dfrac{x^{\frac{n}2}}{(x+\lambda)^\frac{n}2}\,\left\{{}_1\Psi_1^{(\gamma)} 
				      \Big[ \begin{array}{c} \big(\frac{n}{2},2, \tfrac{x+\lambda}2\big) \\
              (\frac{n}{2},1) \end{array}\Big|\, \frac{\lambda x}{(x+\lambda)^2} \Big]\right.\\
					&\qquad\qquad\qquad \left.-\dfrac{\lambda}{x+\lambda}\, {}_1\Psi_1^{(\gamma)} 
					    \Big[\begin{array}{c} \big(\frac{n}{2}+1,2,\tfrac{x+\lambda}2\big)  \\
              (\frac{n}{2}+1,1) \end{array}\Big|\, \frac{\lambda x}{(x+\lambda)^2} \Big]\right\}.
    \end{align*}
\end{theorem}

We express now $S_0$ in terms of the incomplete Fox--Wright function \eqref{IncFoxW1}.

\begin{lemma}\label{lemma1}
For all positive real numbers $p,b>0$ there holds
   \begin{align}\label{FW1}
      S_0(p, b) &= \dfrac1p\,{}_1\Psi_1^{(\gamma)} \Big[ \begin{array}{c} (1,2,p b)  \\
                  (1,1) \end{array}\Big|\, \frac1{4 p^2} \Big]  \\ \label{FW102}
								&= \dfrac1{2p^3}\,{}_1\Psi_1^{(\gamma)} \Big[ \begin{array}{c} (2, 2, p b)  \\
                 (2, 1) \end{array}\Big|\, \frac1{4 p^2}  \Big]- \dfrac{{\rm e}^{-p b}}{p}\, I_0(b)\,.
    \end{align}
\end{lemma}

\begin{proof} Using the definition of the modified Bessel function \eqref{Star} and the special case
of the confluent hypergeometric (Kummer) function \cite[p. 327, Eq. {\bf 13.6.2}]{NIST}
   \[{}_1F_1 \Big[ \begin{array}{c} 1\\2 \end{array} \Big| z\Big]=\dfrac{{\rm{e}}^z-1}{z}, \]
we obtain
   \begin{align*}
      S_0(p, b) &= \int_0^b {\rm e}^{-p x}\, I_0(x)\, {\rm d}x
			           = \sum_{k \geq 0} \dfrac{4^{-k}}{\Gamma(k+1)\, k!}\,
								   \int_0^b {\rm e}^{-p x}\, x^{2k} \, {\rm d}x \\
								&= \sum_{k \geq 0} \dfrac{4^{-k}}{\Gamma(k+1)\, k!}\,
								   \int_0^b \Big( \dfrac{\partial}{\partial p}\Big)^{2k}\, {\rm e}^{-p x} \, {\rm d}x \\
								&= \sum_{k \geq 0} \dfrac{4^{-k}}{\Gamma(k+1)\, k!}\,
								   \Big( \dfrac{\partial}{\partial p}\Big)^{2k}\,\dfrac{1-{\rm e}^{-p b}}p \\
								&= b \sum_{k \geq 0} \dfrac{4^{-k}}{\Gamma(k+1)\, k!}\,
								   \Big( \dfrac{\partial}{\partial p}\Big)^{2k} \,{}_1F_1 \Big[ \begin{array}{c} 1\\2
									 \end{array} \Big| -p b\Big]\,.
	 \end{align*}
As the derivative of the hypergeometric function equals
   \begin{align*}
	    \Big( \dfrac{\partial}{\partial p}\Big)^{2k} \,
		              {}_1F_1 \Big[ \begin{array}{c} 1\\2 \end{array} \Big| -p b\Big]
			         &= \dfrac{(1)_{2k}\,b^{2k}}{(2)_{2k}}\, {}_1F_1 \Big[ \begin{array}{c} 2k+1\\
			            2k+2 \end{array} \Big| -p b\Big] \\
							 &= \dfrac{b^{2k}}{2k+1}\, {}_1F_1 \Big[ \begin{array}{c} 2k+1\\ 2k+2 \end{array} 
								  \Big| -p b\Big] =: H_1\,,
	 \end{align*}
moreover
   \[ H_1 = \dfrac{b^{2k}}{2k+1} \cdot \dfrac{2k+1}{(pb)^{2k+1}}\, \gamma(2k+1, pb)
	        = \dfrac{\gamma(2k+1, pb)}{b\, p^{2k+1}}\,.\]
We have
   \begin{equation} \label{DM1}
	    S_0(p, b) = \dfrac1p \sum_{k \geq 0}\dfrac{\gamma(2k+1, pb)}{\Gamma(k+1)\, k!}\,\dfrac1{(4p^2)^k}\,,
	 \end{equation}
which is equivalent with the first equality in \eqref{FW1}. Finally, having in mind the contiguous 
recurrence formula \cite[p. 178, Eq. {\bf 8.8.1}]{NIST}
   \begin{equation}\label{recc}
	    \gamma(a+1, x) = a\, \gamma(a,x) - x^a\, {\rm e}^{-x}\,,
	 \end{equation}
we transform \eqref{DM1} into an elegant (seemingly new) formula
   \begin{align*}
	    S_0(p, b) &= \dfrac2p \sum_{k \geq 1}\dfrac{\gamma(2k, pb)}{\Gamma(k)\, k!\, (4p^2)^k} -
			             \dfrac{{\rm e}^{-pb}}{p}\, I_0(b) \\ 
								&= \dfrac{1}{2p^3} \sum_{m \geq 0} \dfrac{\gamma(2m+2, pb)}{\Gamma(m+2)\, m!\, (4p^2)^m} -
			             \dfrac{{\rm e}^{-pb}}{p}\, I_0(b)\,, 
	 \end{align*}
which completes the proof of \eqref{FW102}.
\end{proof}

\begin{remark} 
The recursion formula \eqref{recc} should be used with care in computations, because the 
numerical recursion is rather unstable, consult for instance {\rm \cite[p. 114]{Gil}}. 
\end{remark} 

Now, combining Theorem~\ref{theorem1} and Lemma~\ref{lemma1} we infer the following result.

\begin{theorem}
For all $n\in\mathbb{N}$ and $\min\{\lambda,x\}>0$ there holds
   \begin{align*}
      F_{2n,\lambda}(x) &= \dfrac12-{\rm e}^{-\frac{\lambda+x}{2}}
			      \left(\frac12 I_0\left(\sqrt{\lambda x}\right) +
						\sum_{k=1}^{n-1} \left(\dfrac x\lambda \right)^{\frac{k}2}
            I_k(\sqrt{\lambda x})\right)\\
		   &\qquad - \dfrac{|\lambda-x|}{4\lambda x}\,
						{}_1\Psi_1^{(\gamma)} \Big[ \begin{array}{c} \big(1,2,(\sqrt{x}-\sqrt{\lambda})^2/2\big)  \\
            (1,1) \end{array}\Big|\, \frac1{4\lambda x} \Big]\,.
   \end{align*}
Moreover
   \begin{align*}
        F_{2n,\lambda}(x) &= \dfrac12-{\rm e}^{-\frac{x+\lambda}{2}}
			      \left(\frac12 I_0\left(\sqrt{\lambda x}\right) +
						\sum_{k=1}^{n-1} \left(\dfrac x\lambda \right)^{\frac{k}2}
            I_k(\sqrt{\lambda x})\right)\\
			&\qquad-\dfrac{|\lambda-x|}{4\lambda x}\Bigg(\frac1{2\lambda x}\,{}_1\Psi_1^{(\gamma)}
			      \Big[ \begin{array}{c} \big(2,2, \tfrac{(\sqrt{x}-\sqrt{\lambda})^2}2\big)\\ 
						(2,1) \end{array}\Big| \frac1{4\lambda x} \Big] \\
			&\qquad \qquad - {\rm e}^{-\frac{(\sqrt{x}-\sqrt{\lambda})^2}{2}}  
						I_0\Big(\frac{x+\lambda}{2 \sqrt{\lambda x}} -1\Big)\Bigg).
   \end{align*}
\end{theorem}

\section{CDF in terms of incomplete hypergeometric function}\label{sec3}

Besides the already mentioned lower incomplete gamma function \eqref{LIG}, there is also its
complementary function, i.e. upper incomplete gamma function, defined by
   \begin{equation}\label{UIG}
      \Gamma(a,x) = \int_x^\infty{\rm{e}}^{-t}t^{a-1}\,{\rm{d}}t,\qquad \Re(a)>0
	 \end{equation}
and those two functions satisfy the following decomposition formula \cite[p. 174, Eq. {\bf{8.2.3}}]{NIST}
   \[\gamma(a,x)+\Gamma(a,x)=\Gamma(a),\qquad \Re(a)>0.\]
By means of the incomplete gamma functions, in order to generalize the Pochhammer
symbol \cite[p. 22]{Rainville}
   \[ (\lambda)_{\mu} := \dfrac{\Gamma(\lambda+\mu)}{\Gamma(\lambda)}
                       = \begin{cases}
                            1, & \mbox{if}\; \mu = 0;\, \lambda \in \mathbb C \setminus \{0\}\\
                            \lambda(\lambda+1) \cdots (\lambda + n-1), &
														\mbox{if}\; \mu=n \in \mathbb N;\; \lambda \in \mathbb C
                          \end{cases}\,, \]
in 2012 Srivastava {\textit{et al.}} \cite{SriChauAgar} introduced the incomplete Pochhammer symbols
which lead to a natural generalization and decomposition of a class of hypergeometric functions. Precisely,
the incomplete Pochhammer symbols are defined as
   \[(a;x)_\nu := \dfrac{\gamma(a+\nu,x)}{\Gamma(a)}, \qquad [a;x]_\nu :=
	                \dfrac{\Gamma(a+\nu,x)}{\Gamma(a)}, \qquad a,\nu\in\mathbb{C},x \geq 0,  \]
satisfying the decomposition
   \[(a;x)_\nu+[a;x]_\nu = (a)_\nu,\qquad a,\nu \in \mathbb{C},\,x \geq 0.\]
Accordingly, the generalized incomplete hypergeometric functions have the definitions
   \[\!\! {}_p\gamma_q \Big[ \begin{array}{c} (a_1,x), a_2,\dots,a_p\\ b_1,\dots,b_q \end{array}\Big| z \Big]
			 = {}_p\gamma_q \Big[ \begin{array}{c} (a_1, x),\,{\bf a}_{p-1} \\ {\bf b}_q \end{array}\Big| z \Big]
			 = \sum_{n \geq 0} \dfrac{(a_1;x)_n \prod\limits_{j=2}^{p}(a_j)_n}
						   {\prod\limits_{j=1}^{q}(b_j)_n} \dfrac{z^n}{n!}, \]
and
   \begin{equation}\label{IncGamma2}
      \!\!{}_p\Gamma_q \Big[ \begin{array}{c} [a_1,x], a_2,\dots,a_p\\ b_1,\dots,b_q \end{array}\Big| z \Big]
			  = {}_p\Gamma_q \Big[ \begin{array}{c} [a_1, x], {\bf a}_{p-1}\\ {\bf b}_q \end{array}\Big| z \Big]
			  = \sum_{n \geq 0} \dfrac{[a_1; x]_n \prod\limits_{j=2}^{p}(a_j)_n}
						   {\prod\limits_{j=1}^{q}(b_j)_n} \dfrac{z^n}{n!},
   \end{equation}
provided that the infinite series in each case absolutely converges where the appropriate convergence
constraints coincide with the ones for the generalized hypergeometric function ${}_pF_q$, compare
\cite[p. 675, Remark 7]{SriChauAgar}.

One of the most important and widely used cases of generalized incomplete hypergeometric functions
a those when $p=q=1$, the so--called incomplete confluent hypergeometric (or Kummer) function while for
$p=2$, $q=1$ one have the incomplete Gauss hypergeometric function. Regarding to our results, Srivastava
{\textit{et al.}} proved that \cite[p. 680]{SriChauAgar}
   \begin{equation}\label{M2}
      Q_M(\sqrt{2a},\sqrt{2x}) = {\rm e}^{-a}\,{}_1\Gamma_1 \Big[ \begin{array}{c} [M, x]\\
                                 M \end{array}\Big|\, a \Big].
   \end{equation}
Also, it is not difficult to show the connection between the generalized Marcum $Q$--function and
the Gaussian hypergeometric function ${}_2\Gamma_1$, which we realize by the familiar confluence principle.

\begin{lemma} \label{lemma2}
For all $\min\{a,x\}>0$ there holds
   \[ Q_M(\sqrt{2a},\sqrt{2x}) = {\rm e}^{-a} \lim_{b \to \infty} {}_2\Gamma_1\Big[ \begin{array}{c}
			   [M, x], b\\  M \end{array}\Big|\,\frac{a}b \Big]. \]
\end{lemma}

\begin{proof}
Considering the limit representation formula of the Kummer function
   \[	I_\nu(z) = \frac{\left(\frac z2\right)^\nu}{\Gamma(\nu+1)}\,\lim_{b\to\infty}
	               {}_1F_1 \Big[\begin{array}{c}
	                b \\ \nu+1 \end{array} \Big| \frac{z^2}{4b} \Big]\,,\]
the integral representation \eqref{M1} yields
   \begin{align*}
      Q_M(\sqrt{2a},\sqrt{2x}) &= \dfrac{{\rm e}^{-a}}{2^{M-1}\Gamma(M)} \lim_{b\to\infty}
			    \sum_{n \geq 0}\dfrac{(b)_n}{(M)_n\,n!}\left(\dfrac{a}{2b}\right)^n
					\int_{\sqrt{2x}}^\infty t^{2M+2n-1}{\rm{e}}^{-t^2/2}\,{\rm{d}}t\\
       &= \dfrac{{\rm e}^{-a}}{\Gamma(M)} \lim_{b\to\infty}\sum_{n \geq 0}\dfrac{(b)_n}{(M)_n\,n!}
			    \left(\dfrac{a}{b}\right)^n \int_{x}^\infty t^{M+n-1}{\rm{e}}^{-t}\,{\rm{d}}t\\
       &= {\rm e}^{-a} \lim_{b\to\infty} \sum_{n \geq 0}\frac{\Gamma(M+n,x)}{\Gamma(M)}\,
			    \dfrac{(b)_n}{(M)_n\,n!}\,
			    \left(\dfrac{a}{b}\right)^n.
   \end{align*}
By the definition \eqref{IncGamma2} this proves the assertion of the lemma.
\end{proof}

The result \eqref{M2} and the formula proven in the Lemma \ref{lemma2} imply the next limit representation 
formula, which follows by the confluence principle for hypergeometric functions.

\begin{proposition}\label{theorem3}
The following limit representation holds true
    \[\lim_{b \to \infty} {}_2\Gamma_1 \Big[\begin{array}{c} [a, x], b\\c \end{array}\Big|\, \frac{z}b\Big]
		      = {}_1\Gamma_1 \Big[ \begin{array}{c} [a, x]\\ c \end{array}\Big|\, z \Big].\]
\end{proposition}

\begin{proof}
Having in mind the definitions of the generalized incomplete hypergeometric functions \eqref{IncGamma2}
and the upper incomplete gamma function \eqref{UIG}, respectively, we get
   \begin{align*}
      {}_2\Gamma_1 \Big[ &\begin{array}{c} [a, x], b\\ c \end{array}\Big|\, \frac{z}b \Big]
			    = \sum_{n \geq 0}\dfrac{\Gamma(a+n,x)\,(b)_n}{\Gamma(a) (c)_n\,n!}\Big(\dfrac{z}b\Big)^n\\
         &= \int_x^\infty {\rm e}^{-t} \dfrac{t^{a-1}}{\Gamma(a)} \sum_{n \geq 0} \dfrac{(b)_n}{(c)_n\,n!}
				    \left(\dfrac{zt}{b}\right)^n\,{\rm d}t 
					= \int_x^\infty {\rm e}^{-t} \dfrac{t^{a-1}}{\Gamma(a)}\, {}_1F_1 \Big[\begin{array}{c}
	          b \\ c \end{array} \Big| \frac{zt}b \Big]\,{\rm d}t.
   \end{align*}
Now, employing the limit representation \cite{http1} for the confluent hypergeometric function
${}_0F_1$ and the Kummer function ${}_1F_1$, we have
   \[\lim_{b \to \infty} {}_1F_1 \Big[\begin{array}{c} b \\ c \end{array} \Big| \frac{z}b \Big]
	      = {}_0F_1 \Big[\begin{array}{c} - \\ c \end{array} \Big| z \Big]. \]
Therefore, 
   \begin{align*}
      \lim_{b\to\infty} {}_2\Gamma_1 &\Big[ \begin{array}{c} [a, x], b\\ c \end{array}\Big|\, \frac{z}b \Big]
		          = \int_x^\infty  {\rm e}^{-t} \dfrac{t^{a-1}}{\Gamma(a)} \lim_{b \to \infty}
						    {}_1F_1 \Big[\begin{array}{c} b\\ c \end{array} \Big| \frac{zt}b \Big]\,{\rm{d}}t\\
						 &= \int_x^\infty  {\rm e}^{-t} \dfrac{t^{a-1}}{\Gamma(a)}\, {}_0F_1 \Big[\begin{array}{c}
	              - \\ c \end{array} \Big| zt \Big]\,{\rm d}t
						  = \dfrac1{\Gamma(a)} \sum_{n \geq 0} \dfrac{z^n}{(c)_n\,n!}
						    \int_x^\infty{\rm e}^{-t}t^{n+a-1}{\rm{d}}t \\
						 &= \sum_{n \geq 0}\dfrac{\Gamma(a+n,x)}{\Gamma(a)\,(c)_n}\dfrac{z^n}{n!}
						  = \sum_{n \geq 0}\dfrac{[a; x]_n}{(c)_n}\dfrac{z^n}{n!}
              = {}_1\Gamma_1 \Big[ \begin{array}{c} [a, x] \\ c \end{array}\Big|\, z \Big],
   \end{align*}
which completes the proof.
\end{proof}

\begin{remark} Upon setting $x=0$ in {\rm Proposition~\ref{theorem3}} we arrive at the familiar limit
inter--connection representation for the Gaussian and Kummer hypergeometric functions 
{\rm \cite[p. 189]{AAR}}
   \[ \lim_{b\to\infty}{}_2F_1 \Big[\begin{array}{c} a,b \\ c \end{array} \Big| \frac{z}b \Big]
	        = {}_1F_1 \Big[\begin{array}{c} a \\ c \end{array} \Big| z \Big].\]
\end{remark}

\begin{corollary}
The {\rm CDF} of the rv $\xi \sim \chi_\nu'^2(\lambda)$ having $\nu$  degrees of freedom possesses
the representations
   \[ F_{\nu,\lambda}(x) = 1 - {\rm e}^{-\frac\lambda2} {}_1\Gamma_1 \Big[ \begin{array}{c}
	           \big[\tfrac\nu2,\, \tfrac{x}2\big]\\ \tfrac\nu2 \end{array}\Big|\, \frac\lambda2 \Big],\]
and
   \[ F_{\nu,\lambda}(x) = 1 - {\rm e}^{-\frac\lambda2} \lim_{b\to\infty}
	                         {}_2\Gamma_1 \Big[ \begin{array}{c} \big[\tfrac\nu2,\, \tfrac{x}2\big], b  \\
                           \tfrac\nu2 \end{array}\Big|\, \frac\lambda{2 b} \Big].\]
Both formulae are valid for all $\min\{\nu,\lambda,x\}>0$.
\end{corollary}

Finally, it is important to observe that using the expression \eqref{usporedba2} derived in
Section~\ref{sec1} the CDF of the rv $\xi \sim \chi_{2n}'^2(\lambda)$, $n\in\mathbb{N}$, can be
presented explicitly in terms of the lower incomplete Gauss hypergeometric function ${}_2\gamma_1$.

\begin{theorem}
For all $n\in\mathbb N$ and $\min\{\lambda,x\}>0$ there holds
   \begin{align*}
      F_{2n,\lambda}(x) &= \dfrac12-{\rm e}^{-\frac{\lambda+x}2}
			                     \Bigg( \frac12 I_0\big(\sqrt{\lambda x}\big)
				          + \sum_{m=1}^{n-1} \Big(\dfrac x\lambda \Big)^{\frac{m}2} I_m(\sqrt{\lambda x})\Bigg)\\
				  &\qquad - \dfrac{|x-\lambda|}{2(\sqrt{x}+\sqrt{\lambda})^2}\,
					          {}_2\gamma_1 \Big[ \begin{array}{c} (1, (\sqrt{x}+\sqrt{\lambda})^2/2),\, \tfrac12  \\
                    1 \end{array}\Big|\, \dfrac{4 \sqrt{\lambda x}}{(\sqrt{\lambda}+\sqrt{x})^2} \Big].
   \end{align*}
\end{theorem}

\begin{proof}
The modified Bessel function of the first kind can be expressed in terms of the Kummer
confluent hypergeometric function \cite[p. 328, Eq. {\bf 13.6.9}]{NIST}
   \[I_\nu(z) = \dfrac{{\rm e}^{-z}}{\Gamma(\nu+1)}\left(\dfrac{z}{2}\right)^\nu
	              {}_1F_1 \Big[\begin{array}{c} \nu+\tfrac12 \\ 2\nu+1 \end{array} \Big| 2z \Big], \]
and the definition \eqref{functionS} of $S_\nu$ yields
   \begin{align*}
      S_0(\sqrt{\lambda x},\omega) &= \int_0^{\sqrt{\lambda x}} {\rm e}^{-(\omega+2)t}\,
			      {}_1F_1 \Big[\begin{array}{c} \tfrac12 \\ 1 \end{array} \Big| 2t \Big]\,{\rm d}t
					= \sum_{n \geq 0}\dfrac{(\tfrac12)_n\,2^n}{(1)_n\,n!}
					  \int_0^{\sqrt{\lambda x}}t^n{\rm{e}}^{-(w+2)t}\,{\rm d}t\\
         &= \dfrac1{\omega+2} \sum_{n \geq 0} \dfrac{(\tfrac12)_n}{(1)_n\,n!}
						\left(\dfrac2{\omega+2}\right)^n \int_0^{(\omega+2)\sqrt{\lambda x}}u^n{\rm e}^{-u}\,{\rm d}u\\
				 &= \dfrac1{\omega+2} \sum_{n \geq 0} \gamma\Bigg(1+ n,\frac{(\sqrt{x}+\sqrt{\lambda})^2}{2}\Bigg)
				    \dfrac{(\tfrac12)_n}{(1)_n\,n!} \left(\dfrac2{\omega+2}\right)^n\\
				 &= \dfrac1{\omega+2} \sum_{n \geq 0} \dfrac{\big(1;\tfrac{(\sqrt{x}+\sqrt{\lambda})^2}2\big)_n\,
				    (\tfrac12)_n}{(1)_n\,n!} \left(\dfrac2{\omega+2}\right)^n\\
         &= \dfrac1{\omega+2}\, {}_2\gamma_1 \Big[ \begin{array}{c} \big(1, 
				    \tfrac{(\sqrt{x}+\sqrt{\lambda})^2}2\big), \tfrac12 \\ 1 \end{array}\Big|\, 
						\dfrac2{\omega+2} \Big].
   \end{align*}
Now, by virtue of \eqref{usporedba2} the desired representation follows.
\end{proof}

\section{Discussion. Related results. Further remarks}

In this section we discuss some results stated in the introduction, mostly in order to precisely
describe some remaining cases. 

Also, related novel results which concern the complete and incomplete Fox--Wright functions connection 
are obtained. \medskip

\noindent {\sf A.} In the Section~\ref{sec1} we have listed the Theorem~\ref{pomoc} in condensed and 
more elegant form than it is exposed in the original version, namely \cite[p. 4, Theorem 2.1]{DJM_MJOM}
   \[ F_{2n,\lambda}(x) = {\rm e}^{-\frac{\lambda+x}{2}} \left( \sum_{n \geq 0}
	                        \left(\dfrac x\lambda\right)^{\frac{n}2}\,I_n(\sqrt{\lambda x})
												- \sqrt{\dfrac\lambda x}\,\sum_{m=1}^n\left(\dfrac x\lambda\right)^{\frac{m}2}\,
                          I_{m-1}(\sqrt{\lambda x})\right), \]
where $n\in\mathbb{N}$ and $\min\{\lambda,x\}>0$. \medskip

\noindent {\sf B.} The lower an upper incomplete Fox--Wright functions ${}_p\Psi_q[\cdot |_\gamma 
(x, w)]$, and ${}_p\Psi_q[\cdot |^\Gamma (x, w)]$, say, were introduced by Srivastava and 
Pog\'any \cite[pp. 196--197, Eqs. (6) and (7)]{SriPogany} in a study about the generalized 
Voigt--functions. We mention that Srivastava {\it et al.} \cite{Srivastava} considered not only the 
power series definitions, but also the Mellin--Barnes integral forms of the incomplete Fox--Wright 
$\Psi$ functions as well. Similar definition to \cite[p. 131, Eq. (6.1)]{Srivastava} with the detailed 
discussion of the lower incomplete Fox--Wright function is exploited in \cite{MehrezPoganj}. \medskip 

\noindent {\sf C.} One can notice that Temme obtained the formulae \eqref{CDF_Temme} using
the representations of the generalized Marcum $Q$--function \cite[pp. 57--58, Eqs. (2.6) and (2.8)]{Temme},
being a survival function (consult the relation \eqref{distribution}) which codomain is the unit
interval $[0,1]$ and the fact that his results exclude the case $x=\lambda$. Repeating the derivation
procedure applied in \cite[p. 57]{Temme}, in order to present the Marcum $Q$--function in terms of
the integral $T_\nu$, but in our settings for $x=\lambda$ we conclude that {\it mutatis mutandis}
   \[ Q_\mu(a,a) = \dfrac12\,\int_{2a}^\infty {\rm e}^{-t}\,I_{\mu-1}(t)\,{\rm d}t
	               - \dfrac12\,\int_{2a}^\infty {\rm e}^{-t}\,I_{\mu}(t)\,{\rm d}t, \]
which yields in combination with the formula \eqref{distribution} the extended CDF 
formula \eqref{CDF_Temme}, in the upper case expanded into $x \geq \lambda$. Consequently, the CDF formula 
in terms of the finite $S$--integrals can be written in the modified symmetric form which includes 
both \eqref{CDF_Temme2} and \eqref{CDF=}, reads as follows
  \begin{equation} \label{CDF_Temme3}
      F_{n,\lambda}(x) = \begin{cases}
             \dfrac12\left(\dfrac x\lambda\right)^{\frac{n}4}
				     \Big\{ S_{\frac{n}2-1}(\sqrt{\lambda x},\omega) -
						 \sqrt{\dfrac{\lambda}{x}}\, S_{\frac{n}2}(\sqrt{\lambda x},\omega) \Big\} &x \neq \lambda\\
             \dfrac1{\Gamma(\frac{n}2+1)} \Big(\dfrac\lambda2\Big)^{\frac{n}2}\,
						 \Big\{ {}_2F_2 \Big[ \begin{array}{c} \frac{n-1}2, \frac{n}2\\ \frac{n}2+1, n-1 \end{array}
						 \Big| - 2\lambda \Big] \\
					 \qquad - \dfrac\lambda{n+2}\, {}_2F_2 \Big[ \begin{array}{c} \frac{n+1}2,
						 \frac{n}2+1\\ \frac{n}2+2, n+1 \end{array} \Big| - 2\lambda \Big] \Big\}  &x = \lambda	
			 \end{cases}\,.
   \end{equation} 
Also, it is worth to mention that Brychkov has proved the identities \cite[p. 178, Eq. (5)]{Brychkov}
   \begin{equation} \label{Brychkov1}
	    Q_{n+1}(a,a)=\frac12\big(1+{\rm{e}}^{-a^2}I_0(a^2)\big) + {\rm e}^{-a^2}\sum_{k=1}^n I_k(a^2),
	 \end{equation}
and \cite[p. 178, Eq. (7)]{Brychkov}
   \begin{equation} \label{Brychkov2}
	    Q_{n+\frac12}(a,a) = \frac12\big(1+{\rm{erfc}}(\sqrt{2} a)\big)
                        + {\rm e}^{-a^2} \sum_{k=1}^nI _{k-\frac12}(a^2),
	 \end{equation}
valid for any non-negative integer $n$; these formulae also cower the remaining case $x=\lambda$ in
Temme's article. Both Brychkov's displays are important since $F_{n, \lambda}(x)$ is expressed either
for even indices $\mu = n+1$ associated with \eqref{Brychkov1}, or are related to the half-integer case
\eqref{Brychkov2}. \medskip

\noindent {\sf D.} The question about non--negativity of $F_{n, \lambda}(x)$ in
\eqref{CDF_Temme2} was not discussed in detail in the earlier article \cite{DJM_MJOM}.
The non--negativity of expressions in \eqref{CDF_Temme} are evident, being survival functions.
In turn, skipping this approach we can prove the non--negativity of \eqref{CDF_Temme2} by analytical
tools for any $x>0$, using into account certain bounding inequalities for the ratio of
modified Bessel $I$ functions. To do this, start with the input set of $n$ homoscedastic normal
rvs $X_j \sim \mathscr N(\mu_j, \sigma^2)$, which build the quadratic sum rv $\xi$ (see the
introduction) having non--centrality parameter $\lambda>0$, and denote $\mu = \max_{1 \leq j \leq n}
|\mu_j|$. Then, consider the 'normalized' set of rvs $X_j' = X_j \mu^{-1},\, j=\overline{1, n}$ having
term--wise $\mathscr N(\mu_j \mu^{-1}, \sigma^2 \mu^{-2})$ distributions, respectively. The 'normalized'
non--centrality parameter $\lambda' = \lambda \mu^{-2} \leq n$. However, multiplication
with an absolute constant provides the identical structure of the set of input random variables which
defines $\xi$.

Now, assuming $x> \lambda'$ and rewriting the difference of two $S$--terms in
\eqref{CDF_Temme2} in their integral form, then applying Soni's bound \cite[p. 406, Eq. (A)]{Soni}
   \[ I_{\nu+1}(x)< I_\nu(x), \qquad \nu>-\tfrac12,\, x>0, \]
by setting $\nu = \tfrac{n}2-1$, we deduce the estimate
   \begin{align*}
	    J &= \int_0^{\sqrt{\lambda' x}} {\rm e}^{-(\omega+1)t}\, I_{\frac{n}2-1}(t)
			     \Bigg(1-\sqrt{\dfrac{\lambda'}{x}} \dfrac{I_{\frac{n}2}(t)}{I_{\frac{n}2-1}(t)}\Bigg){\rm d}t\\
			  &\geq \int_0^{\sqrt{\lambda' x}} {\rm e}^{-(\omega+1)t} I_{\frac{n}2-1}(t)
			        \Bigg(1-\sqrt{\dfrac{\lambda'}{x}}\, \Bigg){\rm d}t.
	 \end{align*}
Hence, the integral $J$ is obviously positive.

On the other hand, the case $x<\lambda'$ we handle in a similar way, but now with
the aid of the simple functional bound by Joshi and Bissu \cite[p. 255]{Joshi}
   \[ \dfrac{I_{\nu+1}(x)}{I_\nu(x)}< \dfrac{x}{2(\nu+1)}, \qquad \nu>-1,\, x>0.\]
This results in
   \begin{align*}
	    J	\geq \int_0^{\sqrt{\lambda' x}} {\rm e}^{-(\omega+1)t}\, I_{\frac{n}2-1}(t)\,
			         \Big(1-\sqrt{\dfrac{\lambda'}{x}} \dfrac{t}{n}\Big) {\rm d}t
					\geq \Big(1-\dfrac{\lambda'}{n}\Big)\, S_{\frac{n}2-1}\big(\sqrt{\lambda' x}, \omega\big)\,;
	 \end{align*}
the integral $J$ is non--negative as $\lambda' \leq n$. Finally, we point out that the case $x=\lambda$ 
is self-explanatory, compare \eqref{CDF_Temme3}. \medskip 

\section*{Acknowledgments}
The authors are immensely grateful to Professor Nico Temme for his careful reading, constructive
suggestions and helpful comments on an earlier version of the manuscript. His remarks significantly 
contribute to encompass the final version. The research of TKP was partially supported by 
the University of Rijeka, Croatia; project codes {\tt uniri-pr-prirod-19-16} and {\tt uniri-tehnic-18-66}.

\end{document}